\documentclass{elsarticle}
\usepackage{llncsdoc}
\usepackage{amsmath,amssymb}
\newtheorem{theorem}{Theorem}[section]
\newtheorem{lemma}[theorem]{Lemma}
\newtheorem{definition}[theorem]{Definition}

\newtheorem{proposition}[theorem]{Proposition}

\newenvironment{proof}{\vspace*{\parsep}\noindent {\bf proof:}}
{\qed\\[1em]}

\begin{document}

\begin {frontmatter}
\title{Topological Forcing Semantics with Settling}
\author{Robert S. Lubarsky}
\address {Dept. of Mathematical Sciences, Florida Atlantic
University, Boca Raton, FL 33431, USA \\\ead{rlubarsk@fau.edu} }

\begin{abstract} It was realized early on that topologies can
model constructive systems, as the open sets form a Heyting
algebra. After the development of forcing, in the form of
Boolean-valued models, it became clear that, just as over ZF any
Boolean-valued model also satisfies ZF, so do Heyting-valued
models satisfy IZF, which stands for Intuitionistic ZF, the most
direct constructive re-working of the ZF axioms. In this paper, we
return to topologies, and introduce a variant model, along with a
correspondingly revised forcing or satisfaction relation. The
purpose is to prove independence results related to weakenings of
the Power Set axiom. The original motivation is the second model
of \cite {LR}, based on $\mathbb{R}$, which shows that
Exponentiation, in the context of CZF minus Subset Collection,
does not suffice to prove that the Dedekind reals form a set. The
current semantics is the generalization of that model from
$\mathbb{R}$ to an arbitrary topological space. It is investigated
which set-theoretic principles hold in such models in general. In
addition, natural properties of the underlying topological space
are shown to imply the validity of stronger such principles.

\begin {keyword} {constructivism, set theory, semantics, topology\\AMS
classification 03F50, 03E70, 03C90}
\end {keyword}
\end{abstract}

\end {frontmatter}

\section{Introduction}
Topological interpretations of constructive systems were first
studied by Stone \cite{St} and Tarski \cite{T}, who independently
provided such for propositional logic. This was later extended by
Mostowski \cite{Mo} to predicate logic. The first application of
this to any sort of higher-order system was Scott's interpretation
of analysis \cite{Sc1,Sc2}. Grayson \cite{G,G2} then generalized
the latter to the whole set-theoretic universe, to provide a model
of IZF, Intuitionistic Zermelo-Fraenkel Set Theory. Although not
directly relevant to our concerns, it was soon realized that
topological semantics could be unified with Kripke and Beth
models, and all generalized, via categorical semantics; see
\cite{FS} and \cite{MM} for good introductions. Here is
introduced, not a generalized, but rather an alternative
semantics. (Incidentally, this semantics can also be understood
categorically, as determined by Streicher (unpublished).)

An instance of this semantics was already applied in \cite{LR} to
the reals $\mathbb{R}$. The context there was CZF, Constructive
ZF, introduced in \cite {ac1,ac2,ac3} and exposited in \cite
{ac4}. CZF is currently the most studied system of constructive
set theory, because of the modesty of its proof-theoretic strength
coupled with its implicational power. For instance, even though
CZF does not have the proof-theoretically very strong Power Set
Axiom, its substitute, Aczel's Subset Collection, suffices to
prove that the Dedekind reals form a set. What was not clear was
whether a further weakening of Power Set from Subset Collection to
Exponentiation, or the existence of function spaces, would also
prove the same. In the context of the other CZF axioms, Subset
Collection and Exponentiation are proof-theoretically equivalent,
so proof-theoretic analyses would not be able to answer this
question. In contrast, one of the models of \cite {LR} satisfies
CZF$_{Exp}$, yet the Dedekind cuts do not form a set, thereby
showing the necessity of Subset Collection, or at least of
something more than Exponentiation.

The essence of the construction there is that, as in a traditional
topological model, the truth value of set membership ($\sigma \in
\tau$, where $\sigma$ and $\tau$ are terms) is an open set of
$\mathbb{R}$, but at any moment the terms under consideration can
collapse to ground model terms. (A ground model term is the
canonical image of a ground model set -- think of the standard
embedding of V into V[G] in classical forcing.) Such a collapse
does not make the variable sets disappear, though. So no set could
be the Dedekind cuts: any such candidate could at any time
collapse to a ground model set, but then it wouldn't contain the
canonical generic because that's a variable set, and this generic,
over $\mathbb{R}$, is a Dedekind cut. \footnote{For those already
familiar with a similar-sounding construction by Joyal, this is
exactly what distinguishes the two. Joyal started with a
topological space $T$, and took the union of $T$ with a second
copy of $T$, the latter carrying the discrete topology (i.e. every
subset is open). So by Joyal, you could specialize at a point, but
then every set is also specialized there. Here, you can specialize
every set you're looking at at a point, but that won't make the
ambient variable sets disappear. Alternatively, the whole universe
will specialize, but at the same time be reborn. For an exposition
of Joyal's argument in print, see either \cite{G2} or \cite{TvD}
p. 805-807.}

This process of collapsing to a ground model set we call settling
down. Our purpose is to show how this settling semantics works in
an arbitrary topological space, not just $\mathbb{R}$. This
extension is not completely straightforward. Certain uniformities
of $\mathbb{R}$ allowed for simplifications in the definition of
forcing ($\Vdash$) and for proofs of stronger set-theoretic
axioms, most notably Full Separation and Exponentiation. In the
next section, we prove as much as we can making no assumptions on
the topological space $T$ being worked over; in the following
section, natural and appropriately modest assumptions are made on
$T$ so that Separation and Exponentiation can be proven.

The greatest weakness in what can be proven in the general case is
in the family of Power Set-like axioms. This is no surprise, as
the semantics was developed for a purpose which necessitated the
failure of Subset Collection (and hence of Power Set itself). That
Exponentiation ended up holding is thanks to the particularities
of $\mathbb{R}$, not to settling semantics. Rather, what does hold
in general is a weakened version of all of these Power Set-like
axioms. The reason that Power Set fails, like the non-existence of
the set of Dedekind cuts above, is that any candidate for the
power set of $X$ might collapse to a ground model set, and so
would then no longer contain any variable subset of $X$. However,
that variable subset might itself collapse, and then would be in
the classical power set of $X$. So while the subset in question,
before the collapse, might not equal a member of the classical
power set, it cannot be different from every such member. That is
the form of Power Set which holds in the settling semantics:

Eventual Power Set: $\forall X \; \exists C \; (\forall Y \in C \;
Y \subseteq X) \wedge (\forall Y \subseteq X \; \neg \forall Z \in
C \; Y \not = Z)$.

Although we will not need them, there are comparable weakenings of
Subset Collection (or Fullness) and Exponentiation:

Eventual Fullness: $\forall X, Y \; \exists C \; (\forall Z \in C
\; Z$ is a total relation from $X$ to $Y$) $\wedge \; (\forall R$
if $R$ is a total relation from $X$ to $Y$ then $\neg \forall Z
\in C \; Z \not \subseteq R$).

Eventual Exponentiation: $\forall X, Y \; \exists C \; \forall F$
if $F$ is a total function from $X$ to $Y$ then $\neg \forall Z
\in C \; F \not = Z$.

It is easy to see that Power Set implies Fullness, which itself
implies Exponentiation. Essentially the same arguments will prove:

\begin{proposition} Eventual Power Set implies Eventual Fullness, which in
turn implies Eventual Exponentiation.
\end{proposition}

As already stated, the original motivation of this work was to
generalize an extant construction from one to all topologies. Now
that it is done, other uses can be imagined. The theory identified
here is incomparable with CZF, so its proof-theoretic strength is
unclear. If it turns out to be weak, perhaps it could be combined
with CZF to provide a slight strengthening of the latter while
maintaining a similar proof theory. In any case, the
model-theoretic construction might be useful for further
independence results, the purpose of the first, motivating model.
A long-term project is some kind of classification of models,
topological or otherwise; having this unconventional example might
help find other yet-to-be-discovered constructions. A question
raised by van den Berg is how the model would have to be expanded
in order to get a model of IZF. He observed that the recursive
realizability model based on (definable subsets of) the natural
numbers \cite{L} (also discovered independently by Streicher in
unpublished work), which satisfies CZF + Full Separation (and
necessarily not Power Set), is essentially just the collection of
subcountable sets from the full recursive realizability model
\cite{mccarty}, and hence is naturally extendable to an IZF model.
It is at best unclear how the current model could be so extended.
Somewhat speculatively, applications to computer science are also
conceivable, wherever such modeling might be natural. For
instance, constructive logic can naturally be used to model
computation when objects are viewed as having properties only
partially determined at any stage; if in addition parallel
computation is part of the programming paradigm, it could be that
a variable is passed to several parallel sub-computations, which
specify the variable more and in incompatible ways. This is
similar to the current construction, where there are two
transition functions, both leading to the same future but under
one function the variable/generic is fully specified and under the
other it's not.
\section{The General Case}

First we define the term structure of the topological model with
settling, then truth in the model (the forcing semantics), and
then we prove that the model satisfies some standard set-theoretic
axioms.

\begin{definition} For a topological space $T$, a term is a set of the
form $\{ \langle \sigma_i, J_i \rangle \mid i \in I \} \cup \{
\langle \sigma_h, r_h \rangle \mid h \in H \}$, where each
$\sigma$ is (inductively) a term, each $J$ an open set, each $r$
is a member of $T$, and $H$ and $I$ index sets.
\end{definition}

The first part of each term is as usual. It suffices for the
embedding $x \mapsto \hat{x}$ of the ground model into the
topological model:

\begin{definition} $\hat{x} = \{ \langle \hat{y}, T \rangle \mid y \in x
\}$. Any term of the form $\hat{x}$ is called a ground model term.

For $\phi$ a formula in the language of set theory with (set, not
term) parameters $x_0, x_1, ..., x_n$, then $\hat{\phi}$ is the
formula in the term language obtained from $\phi$ by replacing
each $x_i$ with $\hat{x_i}$.

$\check{}$ is the inverse of $\; \hat{}$, for both sets/terms and
formulas: $\hat{\check{\tau}} = \tau$, $\check{\hat{x}} = x$,
$\hat{\check{\phi}} = \phi,$ and $\check{\hat{\phi}} = \phi$.
\end{definition}

The second part of the definition of a term plays a role only when
we decide to have the term settle down and stop changing. This
settling down in described as follows.

\begin{definition} For a term $\sigma$ and $r \in T$,
$\sigma^r$ is defined inductively on the terms as $\{ \langle
\sigma_{i}^r, T \rangle \mid \langle \sigma_{i}, J_{i} \rangle \in
\sigma \wedge r \in J_{i} \} \cup \{ \langle \sigma_h^r, T \rangle
\mid \langle \sigma_h, r \rangle \in \sigma \}$.
\end{definition}

Note that $\sigma^r$ is a ground model term. It bears observation
that $(\sigma^r)^s = \sigma^r$.

\begin{definition}
For $\phi = \phi(\sigma_0, ... , \sigma_i)$ a formula with
parameters $\sigma_0, ... , \sigma_i$, $\phi^r$ is
$\phi(\sigma_0^r, ... , \sigma_i^r)$.
\end{definition}

We define a forcing relation $J \Vdash \phi$, with $J$ an open
subset of $T$ and $\phi$ a formula.

\begin{definition}
$J \Vdash \phi$ is defined inductively on $\phi$:

$J \Vdash \sigma = \tau$ iff for all $\langle \sigma_{i}, J_{i}
\rangle \in \sigma \; J \cap J_{i} \Vdash \sigma_{i} \in \tau$ and
vice versa, and for all $r \in J$ $\sigma^r = \tau^r$

$J \Vdash \sigma \in \tau$ iff for all $r \in J$ there is a
$\langle \tau_{i}, J_{i} \rangle \in \tau$ and $J_r \subseteq J_i$
containing $r$ such that $J_r \Vdash \sigma = \tau_{i}$

$J \Vdash \phi \wedge \psi$ iff $J \Vdash \phi$ and $J \Vdash
\psi$

$J \Vdash \phi \vee \psi$ iff for all $r \in J$ there is a $J_r
\subseteq J$ containing $r$ such that $J_r \Vdash \phi$ or $J_r
\Vdash \psi$

$J \Vdash \phi \rightarrow \psi$ iff for all $J' \subseteq J$ if
$J' \Vdash \phi$ then $J' \Vdash \psi$, and, for all $r \in J$,
there is a $J_r \subseteq J$ containing $r$ such that, for all $K
\subseteq J_r$, if $K \Vdash \phi^r$ then $K \Vdash \psi^r$

$J \Vdash \exists x \; \phi(x)$ iff for all $r \in J$ there is a
$J_r \subseteq J$ containing $r$ and a $\sigma$ such that $J_r
\Vdash \phi(\sigma)$

$J \Vdash \forall x \; \phi(x)$ iff for all  $\sigma$  $J \Vdash
\phi(\sigma)$, and for all $r \in J$ there is a $J_r \subseteq J$
containing $r$ such that for all $\sigma$ $J_r \Vdash
\phi^r(\sigma)$.

\end{definition}
(Notice that in the last clause, $\sigma$ is not interpreted as
$\sigma^r$.)

\begin{lemma} $\Vdash$ is sound for constructive logic.
\end{lemma}

\begin{lemma} \label{equalitylemma2}
T forces the equality axioms, to wit:
\begin{enumerate}
\item $\forall x \; x=x$

\item $\forall x, y \; x=y \rightarrow y=x$

\item $\forall x, y, z \; x=y \wedge y=z \rightarrow x=z$

\item $\forall x, y, z \; x=y \wedge x \in z \rightarrow y \in z$

\item $\forall x, y, z \; x=y \wedge z \in x \rightarrow z \in y.$

\end{enumerate}
\end{lemma}
\begin{proof}

1: It is trivial to show via a simultaneous induction that, for
all $J$ and $\sigma, J \Vdash \sigma = \sigma$, and, for all
$\langle \sigma_{i}, J_{i} \rangle \in \sigma, J \cap J_{i} \Vdash
\sigma_{i} \in \sigma$.

2: Trivial because the definition of $J \Vdash \sigma =_{M} \tau$
is itself symmetric.

3: For this and the subsequent parts, we need a lemma.
\begin{lemma} If $J' \subseteq J \Vdash \sigma =
\tau$ then $J' \Vdash \sigma = \tau$, and similarly for $\in$.
\end{lemma}
\begin{proof} By induction on $\sigma$ and $\tau$.
\end{proof}

Returning to the main lemma, we show that if $J \Vdash \rho =
\sigma$ and $J \Vdash \sigma = \tau$ then $J \Vdash \rho = \tau$,
which suffices. This will be done by induction on terms for all
opens $J$ simultaneously.

For the second clause in $J \Vdash \rho = \tau$, let $r \in J$. By
the hypotheses, second clauses, $\rho^r = \sigma^r$ and $\sigma^r
= \tau^r$, so $\rho^r = \tau^r$.

The first clause of the definition of forcing equality follows by
induction on terms. Starting with $\langle \rho_{i}, J_{i} \rangle
\in \rho$, we need to show that $J \cap J_i \Vdash \rho_i \in
\tau$. We have $J \cap J_{i} \Vdash \rho_{i} \in \sigma$. For a
fixed, arbitrary $r \in J \cap J_{i}$ let $\langle \sigma_{j},
J_{j} \rangle \in \sigma$ and $J' \subseteq J \cap J_{i}$ be such
that $r \in J' \cap J_{j} \Vdash \rho_{i} = \sigma_{j}$. By
hypothesis, $J \cap J_{j} \Vdash \sigma_{j} \in \tau$. So let
$\langle \tau_{k}, J_{k} \rangle \in \tau$ and $\hat{J} \subseteq
J \cap J_{j}$ be such that $r \in \hat{J} \cap J_{k} \Vdash
\sigma_{j} = \tau_{k}$. Let $\tilde{J}$ be $J' \cap \hat{J} \cap
J_{j}$. Note that $\tilde{J} \subseteq J \cap J_{i}$, and that $r
\in \tilde{J} \cap J_{k}$. We want to show that $\tilde{J} \cap
J_{k} \Vdash \rho_{i} = \tau_{k}$. Observing that $\tilde{J} \cap
J_{k} \subseteq J' \cap J_{j}, \hat{J} \cap J_{k}$, it follows by
the previous lemma that $\tilde{J} \cap J_{k} \Vdash \rho_{i} =
\sigma_{j}, \sigma_{j} =_{M} \tau_{k}$, from which the desired
conclusion follows by the induction. So $r \in \tilde{J} \cap
J_{k} \Vdash \rho_{i} \in \tau.$ Since $r \in J \cap J_{i}$ was
arbitrary, $J \cap J_i \Vdash \rho_i \in \tau.$

4: It suffices to show that if $J \Vdash \rho = \sigma$ and $J
\Vdash \rho \in \tau$ then $J \Vdash \sigma \in \tau$. Let $r \in
J$. By hypothesis, let $\langle \tau_i, J_i \rangle \in \tau, J_r
\subseteq J_i$ be such that $r \in J_r \Vdash \rho = \tau_i$;
without loss of generality $J_r \subseteq J$. By the previous
lemma, $J_r \Vdash \rho = \sigma$, and by the previous part of the
current lemma, $J_r \Vdash \sigma = \tau_i$. Hence $J_r \Vdash
\sigma \in \tau$. Since $r \in J$ was arbitrary, we are done.

5: Similar, and left to the reader.
\end{proof}

\begin{lemma} \label{helpfullemma2}
\begin{enumerate}
\item For all $\phi \; \; \emptyset \Vdash \phi$.
\item If $J' \subseteq J \Vdash \phi$ then $J' \Vdash \phi$.
\item If $J_{i}$ $\Vdash \phi$ for all $i$ then $\bigcup_{i} J_{i}
\Vdash \phi$.
\item $J \Vdash \phi$ iff for all $r \in J$ there is a $J_r
\subseteq J$ containing $r$ such that $J_r \Vdash \phi$.
\item For all $\phi, J$ if $J \Vdash \phi$ then for all $r \in J$
there is a neighborhood $J_r$ of $r$ such that $J_r \Vdash
\phi^r$.
\item For $\phi$ bounded (i.e. $\Delta_0$) and having only ground
model terms as parameters, $T \Vdash \phi$ iff $\check{\phi}$
(i.e. V $\models \check{\phi}$).
\end{enumerate}
\end{lemma}

\begin{proof}

1. Trivial induction. This part is not used later, and is
mentioned here only to flesh out the picture.

2. Again, a trivial induction. The base cases, = and $\in$, are
proven by induction on terms, as mentioned just above.

3. By induction. For the case of $\rightarrow$, you need to invoke
the previous part of this lemma. All other cases are
straightforward.

4. Trivial, using 3.

5. By induction on $\phi$.

=: If $r \in J \Vdash \sigma = \tau$ then $\sigma^r = \tau^r$. By
the proof of the first part of the equality lemma, $T \Vdash
\sigma^r = \tau^r$.

$\in$: If $r \in J \Vdash \sigma \in \tau$, let $\tau_i, J_i,$ and
$J_r$ be as given by the definition of forcing $\in$. Inductively,
some neighborhood of $r$ (or, by the previous case, $T$ itself)
forces $\sigma^r = \tau_i^r$. Since $\langle \tau_i^r, T \rangle
\in \tau^r, \; T \Vdash \tau_i^r \in \tau^r$, and $T \Vdash
\sigma^r \in \tau^r$.

$\vee$: If $r \in J \Vdash \phi \vee \psi$, suppose without loss
of generality that $r \in J_r \Vdash \phi$. Inductively let $K_r$
be a neighborhood of $r$ forcing $\phi^r$. Then $K_r \Vdash \phi^r
\vee \psi^r$.

$\wedge$: If $r \in J \Vdash \phi \wedge \psi$, let $J_r$ and
$K_r$ be neighborhoods of $r$ such that $J_r \Vdash \phi$ and $K_r
\Vdash \psi$. Then $J_r \cap K_r$ is as desired.

$\rightarrow$: If $r \in J \Vdash \phi \rightarrow \psi$, then
$J_r$ as given in the definition of forcing $\rightarrow$
suffices. (To verify the second clause in the definition of $J_r
\Vdash \phi^r \rightarrow \psi^r$, use the fact that $(\phi^r)^s =
\phi$ and $(\psi^r)^s = \psi$.)

$\exists$: If $r \in J \Vdash \exists x \; \phi(x)$, let $J_r
\subseteq J$ and $\sigma$ be such that $r \in J_r \Vdash
\phi(\sigma)$. By induction, let $K_r$ be such that $r \in K_r
\Vdash \phi^r(\sigma^r)$. So $K_r \Vdash \exists x \; \phi^r(x)$.

$\forall$: If $r \in J \Vdash \forall x \; \phi(x)$, then $J_r$ as
given by the definition of forcing $\forall$ suffices.

6. A simple induction.
\end{proof}

At this point, we are ready to show what is in general forced
under this semantics.
\begin{theorem} $T$ forces:
\begin{description}
\item [Infinity]
\item [Pairing]
\item [Union]
\item [Extensionality]
\item [Set Induction]
\item [Eventual Power Set]
\item [Bounded ($\Delta_0$) Separation]
\item [Collection]
\end{description}
\end{theorem}

Some comments on this choice of axioms are in order. The first
five are unremarkable. The role of Eventual Power Set was
discussed in the Introduction. The restriction of Separation to
the $\Delta_0$ case should be familiar, as that is also the case
in CZF and KP. By way of compensation, the version of Collection
in CZF is Strong Collection: not only does every total relation
with domain a set have a bounding set (regular Collection), but
that bounding set can be chosen so that it contains only elements
related to something in the domain (the strong version). In the
presence of full Separation, these are equivalent, as an
appropriate subset of any bounding set can always be taken.
Unfortunately, even the additional hypotheses provided by
Collection are not enough in the current context to yield even
this modest fragment of Separation, as will actually be shown at
the beginning of the next section. In fact, even Replacement
fails, as we will see.

\begin{proof}

\begin{itemize}
\item Infinity: $\hat \omega$ will do.
(Recall that the canonical name $\hat{x}$ of any set $x$ from the
ground model is defined inductively as $\{ \langle \hat{y}, T
\rangle \mid y \in x \}.)$
\item Pairing: Given $\sigma$ and $\tau$, $\{\langle \sigma, T \rangle ,
\langle \tau, T \rangle \}$ will do.
\item Union: Given $\sigma$, the union of the following four terms
will do:
\begin{itemize}
\item$\{ \langle \tau, J \cap J_i \rangle
\mid$ for some $\sigma_i, \; \langle \tau, J \rangle \in \sigma_i$
and $\langle \sigma_i, J_i \rangle \in \sigma \}$
\item $\{ \langle \tau, r \rangle \mid$ for some $\sigma_i, \;
\langle \tau, r \rangle \in \sigma_i$ and $\langle \sigma_i, r
\rangle \in \sigma \}$
\item $\{ \langle \tau, r \rangle \mid$ for some $\sigma_i$ and $K, \;
\langle \tau, K \rangle \in \sigma_i, \; r \in K,$ and $\langle
\sigma_i, r \rangle \in \sigma \}$
\item $\{ \langle
\tau, r \rangle \mid$ for some $\sigma_i$ and $K, \; \langle \tau,
r \rangle \in \sigma_i, \; r \in K,$ and $\langle \sigma_i, K
\rangle \in \sigma \}$.
\end{itemize}

\item Extensionality: We need to show that
\begin {equation}
T \Vdash \forall x \; \forall y \; [\forall z \; (z \in x
\leftrightarrow z \in y) \rightarrow x = y].
\end {equation}
It suffices to show that for any terms $\sigma$ and $\tau$,
\begin {equation}
T \Vdash \forall z \; (z \in \sigma \leftrightarrow z \in \tau)
\rightarrow \sigma = \tau.
\end {equation}
(Although that is only the first clause in forcing $\forall$, it
subsumes the second, because $\sigma$ and $\tau$ could have been
chosen as ground model terms in the first place.) To show that,
for the second clause in forcing $\rightarrow$, it suffices to
show that
\begin {equation}
T \Vdash \forall z \; (z \in \sigma^r \leftrightarrow z \in
\tau^r) \rightarrow \sigma^r = \tau^r.
\end {equation}
But, as before, this is already subsumed by choosing $\sigma$ and
$\tau$ to be ground model terms in the first place. Hence it
suffices to check the first clause in forcing $\rightarrow$:
\begin {equation}
\forall J [J \Vdash \forall z \; (z \in \sigma \leftrightarrow z
\in \tau)] \rightarrow [J \Vdash \sigma = \tau].
\end {equation}

To this end, let $\langle \sigma_i, J_i \rangle$ be in $\sigma$;
we need to show that $J \cap J_i \Vdash \sigma_i \in \tau$. By the
choice of $J$, $J \Vdash \sigma_i \in \sigma \leftrightarrow
\sigma_i \in \tau$. In particular, $J \Vdash \sigma_i \in \sigma
\rightarrow \sigma_i \in \tau$. By \ref{helpfullemma2}, part 2),
$J \cap J_i \Vdash \sigma_i \in \sigma \rightarrow \sigma_i \in
\tau$. Since $J \cap J_i \Vdash \sigma_i \in \sigma$ (proof of
\ref{equalitylemma2}, part 1)), $J \cap J_i \Vdash \sigma_i \in
\tau$. Symmetrically for $\langle \tau_i, J_i \rangle \in \tau$.

Also, let $r \in J$. If $\sigma^r \not = \tau^r$, let $\langle
\rho, T \rangle$ be in their symmetric difference. By the choice
of $J$, for some neighborhood $J_r$ of $r$, $J_r \Vdash \rho \in
\sigma^r \leftrightarrow \rho \in \tau^r$. This contradicts the
choice of $\rho$. So $\sigma^r = \tau^r$.

\item Set Induction (Schema): We need to show that
\begin {equation}
T \Vdash \forall x \; ((\forall y \in x \; \phi(y)) \rightarrow
\phi(x)) \rightarrow \forall x \; \phi(x).
\end {equation}
The statement in question is an implication. The definition of
forcing $\rightarrow$ contains two clauses.

The first clause is that, for any open set $J$ and formula $\phi$,
if $J \Vdash \forall x (\forall y \in x \; \phi(y) \rightarrow
\phi(x))$ then $J \Vdash \forall x \; \phi(x)$. By way of proving
that, suppose not. Let $J$ and $\phi$ provide a counter-example.
By hypothesis,
\begin{equation} \label{forall1}
\forall \sigma \; J \Vdash \forall y \in \sigma \; \phi(y)
\rightarrow \phi(\sigma)
\end{equation}
and
\begin{equation} \label{forall2}
\forall r \in J \; \exists J' \ni r \; \forall \sigma' \; J'
\Vdash \forall y \in \sigma' \; \phi^r(y) \rightarrow
\phi^r(\sigma').
\end{equation}
Since J $\not \Vdash \forall x \phi(x)$, either
\begin{equation} \label{notforall1}
\exists \sigma \; J \not \Vdash \phi(\sigma)
\end{equation}
or
\begin{equation} \label{notforall2}
\exists r \in J \; \forall J' \ni r \; \exists \sigma' \; J' \not
\Vdash \phi^r(\sigma').
\end{equation}
If (\ref {notforall2}) holds, let $r$ as given by (\ref
{notforall2}), and then let $J'$ be as given by (\ref {forall2})
for that $r$. By (\ref {notforall2}), $\exists \sigma' \; J' \not
\Vdash \phi^r(\sigma')$; let $\sigma$ be such a $\sigma'$ -- so
$J' \not \Vdash \phi^r(\sigma)$ -- of minimal V-rank. By (\ref
{forall2}), we have $J' \Vdash \forall y \in \sigma \; \phi^r(y)
\rightarrow \phi^r(\sigma).$ If we can show that $J' \Vdash
\forall y \in \sigma \; \phi^r(y)$, then (by the definition of
forcing $\rightarrow$) we will have a contradiction, showing that
(\ref {notforall2}) must fail.

To that end, we must show, unpacking the abbreviation, that $J'
\Vdash \forall y (y \in \sigma \rightarrow \phi^r(y))$; that is,
\begin{equation}
\forall \tau \; J' \Vdash \tau \in \sigma \rightarrow \phi^r(\tau)
\end{equation}
and
\begin{equation}
\forall s \in J' \; \exists K \ni s \; \forall \tau \; K \Vdash
\tau \in \sigma^s \rightarrow \phi^r(\tau),
\end{equation}
the latter because $(\phi^r)^s = \phi^r$.

By way of showing (10), suppose $J' \supseteq K \Vdash \tau \in
\sigma$. Then $K$ can be covered with open sets $K_i$ such that
$K_i \Vdash \tau = \sigma_i$ and $K_i \subseteq J_i$ where
$\langle \sigma_i, J_i \rangle \in \sigma$. Since $\sigma_i$ has
strictly lower V-rank than $\sigma$, $J' \Vdash \phi^r(\sigma_i)$.
Hence $K_i \Vdash \phi^r(\tau)$. Since the $K_i$s cover $K$ (by
lemma \ref{helpfullemma2}, part 3)) $K$ forces the same. We still
have to show that for all $s \in J'$ there is a $K \ni s$ such
that for all $K' \subseteq K$ if $K' \Vdash \tau^s \in \sigma^s$
then $K' \Vdash \phi^r(\tau^s)$. In fact, $J'$ suffices for $K$:
if $J' \supseteq K' \Vdash \tau^s \in \sigma^s$ then $K' \Vdash
\phi^r(\tau^s)$. Moreover, this is the same argument as the one
just completed, with $\sigma^s$ in place of $\sigma$. The only
minor observation that bears making is that the V-rank of
$\sigma^s$ is less than or equal to that of $\sigma$, so again
when $\tau$ is forced to be a member of $\sigma^s$ its V-rank is
strictly less than that of $\sigma$, so the choice of $\sigma$
carries us through.

To show (11), we claim that $J'$ suffices for the choice of $K$:
$J' \Vdash \tau \in \sigma^s \rightarrow \phi^r(\tau)$. Once more,
this is just (10), with $\sigma^s$ in place of $\sigma$.

This completes the proof that (\ref {notforall2}) must fail. Hence
we have that the negation of (\ref {notforall2}) must hold, namely
\begin{equation}
\forall r \in J \; \exists J' \ni r \; \forall \sigma' \; J'
\Vdash \phi^r(\sigma'),
\end{equation}
as well as (\ref {notforall1}). Let $\sigma$ be of minimal V-rank
such that $J \not \Vdash \phi(\sigma)$. If we can show that $J
\Vdash \forall y \in \sigma \; \phi(y)$, then by (\ref {forall1})
we will have a contradiction, completing the proof of the first
clause.

What we need to show are
\begin{equation}
\forall \tau \; J \Vdash \tau \in \sigma \rightarrow \phi(\tau)
\end{equation}
and
\begin{equation}
\forall r \in J \; \exists J' \ni r \; \forall \tau \; J' \Vdash
\tau \in \sigma^r \rightarrow \phi^r(\tau).
\end{equation}
By way of showing (13), suppose $J \supseteq K \Vdash \tau \in
\sigma$; we need to show that $K \Vdash \phi(\tau)$. This is the
same argument, based on the minimality of $\sigma$, as in the
proof of (10). The other part of showing (13) is
\begin{equation}
\forall r \in J \; \exists J' \ni r \; \forall K \subseteq J' \;
(K \Vdash \tau^r \in \sigma^r \Rightarrow K \Vdash
\phi^r(\tau^r)).
\end{equation}
Both (14) and (15) are special cases of (12).

This completes the proof of the first clause.

The second clause is that for all $r \in T$ there is a $J \ni r$
such that for all $K \subseteq J$ if $K \Vdash \forall x \;
((\forall y \in x \; \phi^r(y)) \rightarrow \phi^r(x))$ then $K
\Vdash \forall x \; \phi^r(x)$. For any $r$, let $J$ be $T$. Then
what remains of the claim has exactly the same form as the first
clause, with $K$ and $\phi^r$ for $J$ and $\phi$ respectively.
Since the validity of this first clause was already shown for all
choices of $J$ and $\phi$, we are done.

\item Eventual Power Set: We need to show that
\begin {equation}
T \Vdash \forall X \; \exists C \; \forall Y \; (Y \subseteq X
\rightarrow \neg \forall Z (Z \in C \rightarrow Y \not = Z)).
\end {equation}
(Actually, we must also produce a $C$ that contains only subsets
of $X$. However, to extract such a sub-collection from any $C$ as
above is an instance of Bounded Separation, the proof of which
below does not rely on the current proof. So we will make our
lives a little easier and prove the version of EPS as stated.)
Since the sentence forced has no parameters, the second clause in
forcing $\forall$ is subsumed by the first, so all we must show is
that, for any term $\sigma$,
\begin {equation}
T \Vdash \exists C \; \forall Y \; (Y \subseteq \sigma \rightarrow
\neg \forall Z (Z \in C \rightarrow Y \not = Z)).
\end {equation}

Let $\tau = \{ \langle \hat{x}, r \rangle \mid \sigma^r = \hat{s}
\wedge x \subseteq s \}$. This is the desired $C$. It suffices to
show that
\begin {equation}
T \Vdash \forall Y \; (Y \subseteq \sigma \rightarrow \neg \forall
Z (Z \in \tau \rightarrow Y \not = Z)).
\end {equation}

For the first clause in forcing $\forall$, we need to show that
\begin {equation}
T \Vdash \rho \subseteq \sigma \rightarrow \neg \forall Z (Z \in
\tau \rightarrow \rho \not = Z).
\end {equation}
To do that, first suppose $T \supseteq J \Vdash \rho \subseteq
\sigma$. (Note that that implies that for all $s \in J$ $T \Vdash
\rho^s \subseteq \sigma^s$, so that $\langle \rho^s, s \rangle \in
\tau$, and $T \Vdash \rho^s \in \tau^s$.) We must show that
\begin {equation}
J \Vdash \neg \forall Z (Z \in \tau \rightarrow \rho \not = Z).
\end {equation}
It suffices to show that no non-empty subset $K$ of $J$ forces
\begin {equation}
\forall Z (Z \in \tau \rightarrow \rho \not = Z) \end {equation}
or \begin {equation} \forall Z (Z \in \tau^r \rightarrow \rho^r
\not = Z) \end {equation} ($r \in J$). For the former, we will
show that $K$ must violate the second clause in forcing $\forall$.
Let $s \in K$. Letting $Z$ be $\rho^s$, as just observed, all of
$T$ will force $Z \in \tau^s$ but nothing will force $\rho^s \not
= Z$. Similarly for the latter, by choosing $Z$ to be $\rho^r$. To
finish forcing the implication, it suffices to show that for all
$r$ \begin {equation} T \Vdash \rho^r \subseteq \sigma^r
\rightarrow \neg \forall Z (Z \in \tau^r \rightarrow \rho^r \not =
Z). \end {equation} Again, it suffices to let $Z$ be $\rho^r$.

For the second clause in forcing $\forall$, for $r \in T$ and
$\rho$ a term, it suffices to show that \begin {equation} T \Vdash
\rho \subseteq \sigma^r \rightarrow \neg \forall Z (Z \in \tau^r
\rightarrow \rho \not = Z). \end {equation} This time letting $Z$
by any $\rho^s$ suffices.

\item Bounded Separation: The important point here is that, for
$\phi$ bounded ($\Delta_0$) with only ground model terms, $J
\Vdash \phi$ iff $T \Vdash \phi$ iff $V \models \check{\phi}$
(\ref{helpfullemma2}, part 6).

We need to show that \begin {equation} T \Vdash \forall X \;
\exists Y \; \forall Z \; (Z \in Y \leftrightarrow Z \in X \wedge
\phi(Z)). \end {equation} This means, first, that for any
$\sigma$, \begin {equation} T \Vdash \exists Y \; \forall Z \; (Z
\in Y \leftrightarrow Z \in \sigma \wedge \phi(Z)), \end
{equation} and, second, for any $r \in T$ there is a $J \ni r$
such that, for any $\sigma$, \begin {equation} J \Vdash \exists Y
\; \forall Z \; (Z \in Y \leftrightarrow Z \in \sigma \wedge
\phi^r(Z)). \end {equation} In the second part, choosing $J$ to be
$T$, we have an instance of the first part, so it suffices to
prove the first only.

Let $\tau$ be $\{ \langle \sigma_i, J \cap J_i \rangle \mid
\langle \sigma_i, J_i \rangle \in \sigma$ and $J \Vdash
\phi(\sigma_i) \} \cup \{ \langle \hat{x}, r \rangle \mid \langle
\hat{x}, T \rangle \in \sigma^r$ and $T \Vdash \phi^r(\hat{x})
\}$. We claim that $\tau$ suffices: $T \Vdash \forall Z \; (Z \in
\tau \leftrightarrow Z \in \sigma \wedge \phi(Z))$.

First, let $\rho$ be a term. We need to show that $T \Vdash \rho
\in \tau \leftrightarrow \rho \in \sigma \wedge \phi(\rho)$.
Unraveling the bi-implication and the definition of forcing an
implication, that becomes $J \Vdash \rho \in \tau$ iff $J \Vdash
\rho \in \sigma \wedge \phi(\rho)$, and $J \Vdash \rho^r \in
\tau^r$ iff $J \Vdash \rho^r \in \sigma^r \wedge \phi^r(\rho^r)$.
The first iff should be clear from the first part of the
definition of $\tau$ and the second iff from the second part of
the definition, along with the observation that forcing
$\phi^r(\rho^r)$ is independent of $J$.

We also need, for each $r \in T$, a $J \ni r$ such that for all
$\rho$ $J \Vdash \rho \in \tau^r \leftrightarrow \rho \in \sigma^r
\wedge \phi^r(\rho)$. Choosing $J$ to be $T$ and unraveling as
above (recycling the variable $J$) yields $J \Vdash \rho \in
\tau^r$ iff $J \Vdash \rho \in \sigma^r \wedge \phi^r(\rho)$, and
$J \Vdash \rho^s \in \tau^r$ iff $J \Vdash \rho^s \in \sigma^r
\wedge \phi^r(\rho^s)$. These hold because the only things that
can be forced to be in $\tau^r$ or $\sigma^r$ are (locally) images
of ground model terms, and the truth of $\phi^r$ evaluated at such
a term is independent of $J$.

\item Collection: Since only regular, not strong, Collection is
true here,  it would be easiest to his this with a sledgehammer:
reflect V to some set M large enough to contain all the parameters
and capture the truth of the assertion in question; the term
consisting of the whole universe according to M will be more than
enough. It is more informative, though, to follow through the
natural construction of a bounding set, so we can highlight in the
next section just what goes wrong with the proof of Strong
Collection.

We need \begin {equation} T \Vdash \forall x \in \sigma \; \exists
y \; \phi(x,y) \rightarrow \exists z \; \forall x \in \sigma \;
\exists y \in z \; \phi(x,y). \end {equation} It suffices to show
that for any $J$ \begin {equation} (J \Vdash \forall x \in \sigma
\; \exists y \; \phi(x,y)) \rightarrow (J \Vdash \exists z \;
\forall x \in \sigma \; \exists y \in z \; \phi(x,y)), \end
{equation} and the same relativized to $r$. The latter is a
special case of the former, so it suffices to show just the
former.

By hypothesis, for each $\langle \sigma_i, J_i \rangle \in \sigma$
and $r \in J_i \cap J$ there are $\tau_{ir}$ and $J_{ir} \subseteq
J_i \cap J$, $J_{ir} \ni r$ such that $J_{ir} \Vdash
\phi(\sigma_i, \tau_{ir})$. Also, for all $r \in J$ there is a
$J_r \ni r$ such that, for all $\langle \hat{x}, T \rangle \in
\sigma^r$, $J_r \Vdash \exists y \; \phi^r(\hat{x},y)$. For each
$s \in J_r$, let $\tau_{r\hat{x}s}$ and $K \ni s$ be such that $K
\Vdash \phi^r(\hat{x},\tau_{r\hat{x}s})$. By \ref{helpfullemma2},
part 5), $K \Vdash \phi^r(\hat{x},\tau_{r\hat{x}s}^s)$.

We claim that \begin {equation} \tau = \{ \langle \tau_{ir},
J_{ir} \rangle \mid i \in I, r \in J_i \cap J \} \cup \{ \langle
\tau_{r\hat{x}s}^s, r \rangle \mid r \in J, \langle \hat{x}, T
\rangle \in \sigma^r, s \in J_r \} \end {equation} suffices: $J
\Vdash \forall x \in \sigma \; \exists y \in \tau \; \phi(x,y)$.

Forcing a universal has two parts. The first is that for all
$\rho$, \begin {equation} J \Vdash \rho \in \sigma \rightarrow
\exists y \in \tau \; \phi(\rho,y). \end {equation} For the
second, it suffices to show that for all $r \in J$ and terms
$\rho$ \begin {equation} J_r \Vdash \rho \in \sigma^r \rightarrow
\exists y \in \tau^r \; \phi^r(\rho,y). \end {equation}

For the former, first suppose $J \supseteq K \Vdash \rho \in
\sigma$. It should be clear that the first part of $\tau$ covers
this case. For the other part of forcing that implication, for
each $r \in J$, it suffices to show that $J_r$ is as desired: for
all $K \subseteq J_r$, if $K \Vdash \rho^r \in \sigma^r$ then $K
\Vdash \exists y \in \tau^r \; \phi^r(\rho^r,y)$. This is subsumed
by the second implication from above, to which we now turn.

To show $J_r \Vdash \rho \in \sigma^r \rightarrow \exists y \in
\tau^r \; \phi^r(\rho,y)$, we need to show first that if $J_r
\supseteq K \Vdash \rho \in \sigma^r$ then $K \Vdash \exists y \in
\tau^r \; \phi^r(\rho,y)$, and second that for all $s \in J_r$
there is a $K \ni s$ such that if $K \supset L \Vdash \rho^s \in
\sigma^r$ then $L \Vdash \exists y \in \tau^r \;
\phi^r(\rho^s,y)$. By choosing $K$ to be $J_r$, the second is
subsumed by the first. For that, it should be clear that the
second part of $\tau$ covers this case. In a bit more detail, it
suffices to work locally. (That is, it suffices to find a
neighborhood of $s \in K$ forcing what we want, by lemma
\ref{helpfullemma2}.) Locally, $\rho$ is forced equal to some
$\hat{x}$, where $\langle \hat{x}, T \rangle \in \sigma^r$. As
already shown, some neighborhood of $s$ forces
$\phi^r(\hat{x},\tau_{r\hat{x}s}^s)$, and $\langle
\tau_{r\hat{x}s}^s, T \rangle \in \tau^r$ by the second part of
$\tau$.

\end{itemize}
\end{proof}

\section{Exponentiation, Separation, and Replacement}
The second model of \cite{LR} is the topological semantics of the
current paper applied to $\mathbb{R}$ (with the standard
topology). There it was shown that the model satisfies not just
the axioms proven here but also Exponentiation and Separation, and
hence, in the presence of Collection, Replacement too. The reason
those extra properties hold in that case is that $\mathbb{R}$ is a
``nice" space. It is the purpose of this section to explore just
what makes $\mathbb{R}$ nice and why such niceness is necessary
for these additional properties.

\subsection{Exponentiation}
We can identify exactly the property of $T$ that would make
Exponentiation hold.

\begin{theorem} $T \Vdash$ Exponentiation iff $T$ is locally
connected.
\end{theorem}
\begin{proof}
First we do the right-to-left direction. So suppose $T$ is locally
connected. Given terms $\sigma$ and $\chi$, let $\tau$ be $\{
\langle \rho, J \rangle \mid J \Vdash \rho$ is a function from
$\sigma$ to $\chi \} \cup \{ \langle \hat{x}, r \rangle \mid x$ is
a function from $\check{\sigma^r}$ to $\check{\chi^r} \}$. ($\tau$
can be arranged to be set-sized by requiring that $\rho$ be
hereditarily empty outside of $J$.) It suffices to show that $T
\Vdash \forall z \; (z \in \tau \leftrightarrow z$ is a function
from $\sigma$ to $\chi)$.

The first clause in forcing $\forall$ is that, for any term
$\rho$, $T \Vdash \rho \in \tau \leftrightarrow \rho$ is a
function from $\sigma$ to $\chi$. That $J \Vdash \rho \in \tau$
iff $J \Vdash ``\rho$ is a function from $\sigma$ to $\chi"$ is
immediate from the first part of $\tau$. As for $J \Vdash \rho^r
\in \tau^r$ iff $J \Vdash ``\rho^r$ is a function from $\sigma^r$
to $\chi^r"$, by \ref{helpfullemma2}, part 6), both of those
statements are independent of $J$, and the iff holds because of
the second part of $\tau$.

The crux of the matter is the second clause in forcing $\forall$:
$J \Vdash \rho \in \tau^r$ iff $J \Vdash ``\rho$ is a function
from $\sigma^r$ to $\chi^r"$. Why can only ground model functions
be forced (locally) to be functions? For $s \in J$, let $K_s
\subseteq J$ be a connected neighborhood of $s$. For each $\langle
\sigma_i, T \rangle \in \sigma^r$, pick a $\langle \chi_i, T
\rangle \in \chi$ such that the value of (i.e. the largest subset
of $K_s$ forcing) $``\rho(\sigma_i) = \chi_i"$ is non-empty. That
set, along with the value of $``\rho(\sigma_i) \not = \chi_i"$, is
a disjoint open cover of $K_s$. Since $K_s$ is connected, the
latter set is empty. So all of the values of $\rho$ are determined
by $K_s$, so $K_s$ forces $\rho$ to equal a ground model term.
Since $J$ is covered by such sets, $J$ also forces $\rho$ to be a
ground model term.

Now for the converse, suppose $T$ is not locally connected.
Assuming that $T$ still forces Exponentiation, we will come up
with a contradiction. Let $x \in J \subseteq T$ be such that no
neighborhood of $x$ which is a subset of $J$ is connected. Working
within $J$ as a subspace of $T$, $J$ is itself not connected, and
so can be partitioned into two clopen subsets, $K_0$ and $J_0 \ni
x$. Inductively, given $x \in J_n$ clopen, partition $J_n$ into
clopen $K_{n+1}$ and $J_{n+1} \ni x$. Let $O$ consist of all of
the $J_i$s. By Exponentiation, $J$ is covered by sets $I$ forcing
$``Z_I$ is the function space from $\hat{O}$ to $\hat{2}$". In
particular, for each $r \in J$ there is an $I_r$ with $r \in I_r
\subseteq I$ such that, for each $\sigma$, $I_r \Vdash ``\sigma
\in Z_I^r \leftrightarrow \sigma$ is a function from $\hat{O}$ to
$\hat{2}$." Notice that each $Z_I^r$ consists only of ground model
terms standing for ground model functions, and each ground model
function from $O$ to $2$ is forced by $J$ to be such a function
from $\hat{O}$ to $\hat{2}$, so each $Z_I^r$ equals $\hat{O^2}$.
In short, for each $\sigma$, $J \Vdash \sigma \in \hat{O^2}
\leftrightarrow \sigma$ is a function from $\hat{O}$ to
$\hat{2}$."

Now let $f$ be the term such that $J_i \Vdash f(\hat{J}_i) = 1$
and $K_i \Vdash f(\hat{J}_i) = 0$." $J \Vdash ``f$ is a function
from $\hat{O}$ to $\hat{2}$," so $J \Vdash ``f \in \hat{O^2}".$
Each point in $J_\omega := \bigcap_i J_i$ has a neighborhood,
necessarily a subset of $J_\omega$, forcing $f$ to be the constant
function 1, so $J_\omega$ is open. As an intersection of closed
sets, $J_\omega$ is also closed, and contains $x$ to boot. This
construction can continue indefinitely: at stage $\alpha + 1$,
split the clopen $J_\alpha \ni x$ into clopen $K_{\alpha + 1}$ and
$J_{\alpha + 1} \ni x$; at stage $\lambda$ a limit, consider the
function space from $\{ J_\alpha \mid \alpha < \lambda \}$ to 2.
This is a contradiction, because it produces class-many subsets of
$J$.
\end{proof}

An example of a non-locally connected space is Cantor space.
Forcing with that produces a random 0-1 sequence, which is a
function from $\mathbb{N}$ to 2. So the canonical generic is in a
function space, but cannot be captured by any ground model set.

\subsection {Separation}
The situation here seems more difficult than for Exponentiation,
because we have not yet been able to find a property on $T$
equivalent to Separation. Indeed, it is questionable whether there
is any such nice property, as discussed at the end of this
sub-section. Nevertheless, we still have a theorem and some
examples.

It is instructive to see why, in the proof of Separation in the
main theorem, Full Separation did not hold, only Bounded. The
problem came with the settling. Given $\sigma$ and $\langle
\hat{x}, T \rangle \in \sigma^r$, we need to know whether to put
$\hat{x}$ into the subset $\tau$ of $\sigma$ defined by $\phi$. We
can certainly look for a neighborhood forcing $\phi^r(\hat{x})$ or
its negation. But when forcing a universal, we need a fixed
neighborhood $J_r$ of $r$ deciding each $\phi^r(\hat{x})$
simultaneously, and cannot afford to use a separate $J_{r
\hat{x}}$ for each different $\hat{x}$. Since all the parameters
in that formula are ground model terms, it is not their meanings
that could change over different open sets, but rather only the
topology itself and what it makes true and false. So the natural
hypothesis to say that this doesn't happen is that, locally, all
points look alike.

\begin{definition}
$T$ is locally homogeneous around $r, s \in T$ if there are
neighborhoods $J_r, J_s$ of $r$ and $s$ respectively and a
homeomorphism of $J_r$ to $J_s$ sending $r$ to $s$.

An open set $U$ is homogeneous if it is locally homogeneous around
all $r, s \in U$.

$T$ is locally homogeneous if every $r \in T$ has a homogeneous
neighborhood.
\end{definition}

\begin{lemma}
If $U$ is homogeneous, $\phi$ contains only ground model terms,
and $U \supseteq V \Vdash \phi \;(V$ non-empty$)$, then $U \Vdash
\phi$.
\end{lemma}

\begin{proof}
Let $r \in V$. For $s \in U$, let $V_r$ and $V_s$ be the
neighborhoods $f$ the homeomorphism given by the homogeneity of
$U$. $f(\sigma)$ can be defined inductively on terms $\sigma$.
(Briefly, hereditarily restrict $\sigma$ to $V_r$ and apply $f$ to
the second parts of the pairs in the terms.) $f(\psi)$ is then
$\psi$ with $f$ applied to the parameters. It is easy to show
inductively on formulas that $V_r \Vdash \psi$ iff $V_s \Vdash
f(\psi)$.

If $\phi$ contains only ground model terms, then $f(\phi) = \phi$.
So $U$ is covered by open sets that force $\phi$. Hence $U \Vdash
\phi$.
\end{proof}

\begin{theorem} \label {homog-sep}
If $T$ is locally homogeneous then $T \Vdash Full Separation$.
\end{theorem}
\begin{proof}
As in the proof of Bounded Separation from the previous section,
we have to show that, for any $\sigma$, $T \Vdash \exists Y \;
\forall Z \; (Z \in Y \leftrightarrow Z \in \sigma \wedge
\phi(Z))$, only this time with no restriction on $\phi$. The
choice of witness $Y$ is slightly different. For each $r$ let $K_r
\ni r$ be homogeneous. Let $\tau$ be $\{ \langle \sigma_i, J \cap
J_i \rangle \mid \langle \sigma_i, J_i \rangle \in \sigma$ and $J
\Vdash \phi(\sigma_i) \} \cup \{ \langle \hat{x}, r \rangle \mid
\langle \hat{x}, T \rangle \in \sigma^r$ and $K_r \Vdash
\phi^r(\hat{x}) \}$. The difference from before is that in the
latter part of $\tau$ membership is determined by what's forced by
$K_r$ instead of by $T$. We claim that $\tau$ suffices: $T \Vdash
\forall Z \; (Z \in \tau \leftrightarrow Z \in \sigma \wedge
\phi(Z))$.

For the first clause in forcing $\forall$, let $\rho$ be a term.
We need to show $T \Vdash \rho \in \tau \leftrightarrow \rho \in
\sigma \wedge \phi(\rho)$. By the first clause in forcing
$\rightarrow$, we have to show that for all $J \; J \Vdash \rho
\in \tau$ iff $J \Vdash \rho \in \sigma \wedge \phi(\rho)$, which
should be clear from the first part of $\tau$. For the second
clause in $\rightarrow$ it suffices to show that for all $J
\subseteq K_r \; J \Vdash \rho^r \in \tau^r$ iff $J \Vdash \rho^r
\in \sigma^r \wedge \phi^r(\rho^r)$. Regarding forcing membership,
all of the terms here are ground model terms, so membership is
absolute (does not depends on the choice of $J$). If $\rho^r$
enters $\tau^r$ because of the first part of $\tau$'s definition,
then we have $\sigma_i^r = \rho^r, \; r \in J \Vdash
\phi(\sigma_i), \; r \in J_i$, and $\langle \sigma_i, J_i \rangle
\in \sigma$. By \ref{helpfullemma2}, part 5), some neighborhood
$J_r$ of $r$ forces $\phi^r(\sigma_i^r)$. By the lemma just above
(applied to $K_r \cap J_r$), $K_r$ forces the same. Hence we can
restrict our attention to terms $\rho^r$ which enter $\tau$
because of $\tau$'s definition's second part. Again by the
preceding lemma, for $J$ non-empty, $J \Vdash \phi^r(\rho^r)$ iff
$K_r \Vdash \phi^r(\rho^r)$, which suffices. (For $J$ empty, $J$
forces everything.)

For the second clause in forcing $\forall$, it suffices to show
that $K_r \Vdash \rho \in \tau^r \leftrightarrow \rho \in \sigma^r
\wedge \phi^r(\rho)$. If any $J \subseteq K_r$ forces $\rho \in
\tau^r$ or $\rho \in \sigma^r$, then locally $\rho$ is forced to
be some ground model term, and we're in the same situation as in
the previous paragraph.
\end{proof}

We would like to see to what degree we can turn the previous
theorem into an iff. Toward that end, suppose $T$ is not locally
homogeneous. So we can choose $r \in T$ which has no homogeneous
neighborhood. That means every neighborhood of $r$ contains two
points, say $s$ and $t$, with no local homeomorphism sending $s$
to $t$. If there were local homeomorphisms from $r$ to both $s$
and $t$, they could be composed to get one from $s$ to $t$. So one
of $s$ and $t$ can be chosen to be $r$.

Example 1: It is possible that there is a fixed $s$ that can be
chosen as a witness to $r$'s non-homogeneity from every
neighborhood of $r$. In particular, every open containing $r$ also
contains $s$. If every open set containing $s$ also contained $r$,
then the function interchanging $r$ and $s$ and leaving everything
else alone would be a homeomorphism contradicting the assumptions,
so some neighborhood of $s$ does not contain $r$. The smallest
example of such a space is the two-element space $T = \{r, s\}$,
with opens $T, \emptyset,$ and $\{s\}$. Let $\sigma$ be $\{
\langle 0, r \rangle\}$; that is, 0 gets into $\sigma$ when we
settle at $r$. Let $\phi$ be $\forall x, y \; x \not = y \vee \neg
x \not = y.$ Suppose $Z$ were $\{ x \in \sigma \mid \phi \}$:
\begin {equation} T \Vdash ``\forall x \; x \in Z \leftrightarrow x \in
\sigma \wedge \phi". \end {equation} In particular, settling at $r$, we
get \begin {equation} T \Vdash ``\forall x \; x \in Z^r
\leftrightarrow x = 0 \wedge \phi^r", \end {equation} or, more
simply, \begin {equation} T \Vdash 0 \in Z^r \leftrightarrow \phi.
\end {equation} As $Z^r$ is a ground model term, $T$ decides $0 \in Z^r$.
But $T$ does not decide $\phi$: $T$ does not force ``$\sigma \not = 1
\vee \neg \sigma \not = 1"$, but $\{s\}$ forces $\phi$, both of
which can be calculated by cranking through the definitions. Hence
$T$ does not force Separation.

Example 2: An example of the opposite kind is where there is no
$s$ in every open neighborhood of $r$. Here let $T$ be
$\mathbb{R}^{\geq 0}$. Let $\sigma$ be $\{\langle 0,0 \rangle\}$.
Let $\phi$ be \begin {equation} \forall x \subseteq 1 \; \exists y
\; ((0 \in x \rightarrow y = 0 \vee y = 1) \wedge (\neg \neg y = 0
\vee \neg \neg y = 1 \rightarrow 0 \in x)). \end {equation} If
Separation held, we could form $\{ x \in \sigma \mid \phi\}$.
Settling at 0, we would have $\{ 0 \mid \phi\}$ as a ground model
set. But $\phi$ is not decided in any neighborhood of 0. That's
because $\mathbb{R}^{>0} \Vdash \phi$, since $y$ can be chosen to
alternate between 0 and 1 on the disjoint intervals that
constitute the support of $x$. But on any neighborhood containing
0, instantiating $x$ with $\{\langle 0, \mathbb{R}^{>0} \rangle\}$
forces any such $y$ to be the constant 0 or 1 on the positives,
hence not not equal to 0 or not not equal to 1, but does not force
$x$ to be inhabited.

What these examples indicate is less that the failure of
homogeneity leads to the failure of Separation, but rather that
what is needed in such a construction is the transferability of a
property of the underlying topology into the internal language of
the set theory. More explicitly, all the proof of Separation
needed was a $K_r \ni r$, which may well depend on the choice of
$\phi$ and $\sigma$, such that, for all $\langle \hat{x}, T
\rangle \in \sigma^r, \; K_r \Vdash \phi^r(\hat{x})$ or $K_r
\Vdash \neg \phi^r(\hat{x})$. It is easy to see that the existence
of $r, \phi,$ and $\sigma$ for which there is no such $K_r$
immediately provides a counter-example to Separation. This is
apparently less than the homeomorphisms needed for local
homogeneity. We find it unlikely that there would be a direct
correspondence between any natural topological property and this
feature of the forcing, so closely tied to set theory and its
language. In contrast, there could well be such a topological
property for a certain formula or class of formulas. It would be
interesting to know to what degree this is possible, and why, even
when the answer is simply ``not at all."

\subsection {Replacement and Strong Collection}
If Separation were to hold (in the presence of the other axioms
from above), then Strong Collection would follow, which itself
implies Replacement. Hence a powerful way to show that Separation
is not forced is to give an example in which even Replacement
fails. In the example below, the offending formula is a Boolean
combination of $\Sigma_1$ formulas. This is about as strong a
result as one could hope for, as further restrictions on the
formula render Replacement provable. Suppose, for instance, the
function were $\Sigma_1$ definable: $\forall x \in A \; \exists !
y \; \exists z \; \phi(x,y,z)$, where $\phi$ is $\Delta_0$. By
Collection, there is a bounding sets for triples $\langle x,y,z
\rangle$ with $\phi(x,y,z)$, as $x$ runs through $A$. By
$\Delta_0$ Separation, we can restrict that bounding set to only
triples where the first component is in $A$ and the triple
satisfies $\phi$. Then by $\Delta_0$ Separation again, we can
project onto the second coordinate. Perhaps there is still some
wiggle room between $\Sigma_1$ and Boolean combinations thereof,
such as in the number and use of implications (including
negations), if one wanted to fine tune this result, but there's
certainly not a lot.

Example 3: Let $T_n$ ($n>0$) be the standard space for collapsing
$\aleph_n$ to be countable: elements are injections from
$\aleph_0$ to $\aleph_n$, an open set is given by a finite partial
function of the same type, an element is in that open set if it is
compatible with the partial function. Let $T$ be the disjoint
union of the $T_n$s adjoined with an extra element $\infty$:
$\biguplus_n T_n \cup \{\infty \}$. A basis for the topology is
given by all the open subsets of each $T_n$, plus the basic open
neighborhoods of $\infty$, which are all of the form $\biguplus_{n
\geq N} T_n \cup \{\infty \}$ for some $N$.

This $T$ falsifies Replacement. To state the instance claimed to
be falsified, we need several parameters. One is $\{ \langle
\hat{n}, \infty \rangle \mid n \in \omega \}$, which we will call
$\omega^-$. Another is the internalization of the function $n
\mapsto \aleph_n \; (n \in \omega)$, which we will refer to via
the free use of the notation $\hat{\aleph}_n$, even when $n$ is
just a variable. Finally, we will implicitly need $\hat{\omega}$
in the assertion ``$X$ is countable," which is the abbreviation
for what you think (i.e. the existence of a bijection with
$\hat{\omega}$). Note that ``$X$ is uncountable" is taken as the
negation of ``$X$ is countable."

\begin {definition} Let $\phi(x,z)$ be the conjunction of:
\begin {enumerate}
\item $z=0 \vee z=1$
\item $z = 0 \leftrightarrow \hat{\aleph}_x$ is uncountable
\item $z = 1 \leftrightarrow \neg\neg\hat{\aleph}_x$ is countable.
\end {enumerate}
\end {definition}

\begin{proposition}
$T \not \Vdash \forall x \in \omega^- \; \exists ! y \; \phi(x,y)
\; \rightarrow \exists f \; \forall x \in \omega^- \phi(x, f(x))$.
\end{proposition}

\begin{proof}
First we show that $T$ forces the antecedent $\forall x (x \in
\omega^- \rightarrow \exists ! y \; \phi(x,y))$.

For the first clause in forcing $\forall$, we need to show that
for all $\sigma$ $T \Vdash \sigma \in \omega^- \rightarrow \exists
! y \; \phi(\sigma,y)$. The first clause in forcing that
implication is vacuous, as no open set will force $\sigma \in
\omega^-$. The second clause is vacuous for all choices of $r$
except $\infty$, as then $(\omega^-)^r$ is empty. Finally, for $r
= \infty$, it suffices to show that $T \Vdash \exists ! y \; [(y =
0 \vee y = 1) \wedge (y = 0 \leftrightarrow
\hat{\aleph}_{\hat{n}}$ is uncountable$) \wedge (y = 1
\leftrightarrow \neg\neg\hat{\aleph}_{\hat{n}}$ is countable$)]$.
The term which is $0$ on $\biguplus_{0<i<n} T_n$ and $1$ on the
rest of $T$ suffices.

The second clause in forcing $\forall$ is similar.

Since $T$ forces the antecedent of the conditional, it suffices to
show that $T$ does not force the consequent: $T \not \Vdash
\exists f \; \forall x \in \omega^- \; \phi(x, f(x))$. If that
were not the case, there would be a term (we will ambiguously
refer to as $f$) and a neighborhood $J$ of $\infty$ such that $J
\Vdash \forall x \in \omega^- \; \phi(x, f(x))$. By lemma
\ref{helpfullemma2}, part 5), there would be a $K \ni \infty$ such
that $K \Vdash \forall x \in \hat{\omega} \; \phi(x,
f^\infty(x))$. K, being open, contains a set of the form
$\biguplus_{n \geq N} T_n$. Let $M$ be $N+1$. So $K \Vdash
\phi(\hat{M}, (f^{\infty}(\hat{M}))$. But $f^{\infty}(\hat{M})$ is
a ground model term, and so is (forced by $K$ to be) equal to
$\hat{0}$ or $\hat{1}$. Hence either $K \Vdash
\hat{\aleph}_{\hat{M}}$ is uncountable or $K \Vdash \neg \neg
\hat{\aleph}_{\hat{M}}$ is countable. But neither is the case,
since $K \supseteq T_N \Vdash \hat{\aleph}_{\hat{M}}$ is
uncountable and $K \supseteq \biguplus_{n>N}T_n \Vdash
\hat{\aleph}_{\hat{M}}$ is countable.
\end {proof}

Since the preceding is an example where Separation fails, by the
results of the previous sub-section, local homogeneity must fail
too. Arguably, though, that's not the essence of the construction.
In this case, what determined the choice of $y$ was a different
open set for each $x$. No given neighborhood of $\infty$ sufficed,
because it would have been split for some $x$ into a
sub-neighborhood forcing 0 and another forcing 1. So it seems to
be a matter of connectedness.

\begin {theorem} If T is locally connected then T $\Vdash$
Replacement.
\end {theorem}

\begin {proof}
Sketch of proof: When showing the existence of a good bounding set
(cf. the proof of Collection in the main theorem), work in a
connected neighborhood $J_r$ of any given point $r$. For every $J
\Vdash \sigma_i \in \sigma \wedge \phi(\sigma_i, \tau_i)$, include
$\langle \tau_i, J \cap J_r \rangle$ in the bounding set $\tau$.
As for the settling, for any $\langle \hat{x}, T \rangle \in
\sigma^r$, by the totality of $\phi$, $J_r$ is covered by sets $K$
forcing $\phi(\hat{x}, \tau_K)$, for some choice of $\tau_K$. By
settling, $\tau_K$ can be taken to be ground model terms
$\hat{y}_K$. By uniqueness, the $\hat{y}_K$s have to agree
wherever the $K$s overlap. Since they're ground model terms, they
don't vary, so are the same ground model terms on all $K$s that
overlap. By the connectedness of $J_r$, all the $\hat{y}_K$s are
equal, say $\hat{y}$. Include $\langle \hat{y}, r \rangle$ in
$\tau$.
\end {proof}

Analogously with Separation, we do not believe that there is an
equivalence between local connectedness and Replacement. Rather,
it's likely that what's at stake is some kind of definitional
connectedness, whether an open set can be split into disjoint
clopens that are the truth values of different statements. It
would be nice to see a nice topological equivalent of that
property or any interesting fragment of it.

Finally, it would be of interest to see any examples or theorems
along these lines pertaining to Strong Collection that cannot be
reduced to ones about Replacement or Separation.


\begin{thebibliography}{99}
\bibitem{ac1} Aczel, P.: The type theoretic interpretation of
constructive set theory. In MacIntyre, A., Pacholski, L., Paris,
J. (eds.) Logic Colloquium '77 pp. 55-66. North-Holland, Amsterdam
(1978)
\bibitem{ac2} Aczel, P.: The type theoretic interpretation of
constructive set theory: choice principles. In Troelstra, A.S.,
van Dalen, D. (eds.) The L.E.J. Brouwer Centenary Symposium, pp.
1-40. North-Holland, Amsterdam (1982)
\bibitem{ac3} Aczel, P.: The type theoretic interpretation of
constructive set theory: inductive definitions. In Marcus, R.B. et
al. (eds.) Logic, Methodology and Philosophy of Science VII pp.
17-49. North-Holland, Amsterdam (1986)
\bibitem{ac4} Aczel, P., Rathjen, M.: Notes on constructive set
theory, Technical Report 40, 2000/2001. Mittag-Leffler Institute,
Sweden (2001); forthcoming as book
\bibitem{FS} Fourman, M.P., Scott, D.S.: Sheaves and Logic. In
Fourman, M.P., Mulvey, C.J., Scott, D.S. (eds.) Applications of
Sheaves, Lecture Notes in Mathematics Vol. 753, pp. 302-401.
Springer-Verlag, Berlin Heidelberg New York (1979)
\bibitem{G} Grayson, R.J.: Heyting-valued models for intuitionistic
set theory. In Fourman, M.P., Mulvey, C.J., Scott, D.S. (eds.)
Applications of Sheaves, Lecture Notes in Mathematics Vol. 753,
pp. 402-414. Springer-Verlag, Berlin Heidelberg New York (1979)
\bibitem{G2} Grayson, R.J.: Heyting-valued semantics. In Lolli,
G., Longo, G., Marcja, A. (eds.) Logic Colloquium '82, Studies in
Logic and the Foundations of Mathematics Vol. 112, pp. 181-208.
North-Holland, Amsterdam New York Oxford (1984)
\bibitem{L} Lubarsky, R.: CZF and Second Order Arithmetic,
Annals of Pure and Applied Logic 141, pp. 29-34 (2006)
\bibitem{LR} Lubarsky, R., Rathjen, M.: On the constructive Dedekind
reals. In Artemov, S.N., Nerode, A. (eds.) Proceedings of LFCS
'07, LNCS Vol. 4514, pp. 349-362. Springer (2007). Also Logic and
Analysis 1, pp. 131-152 (2008)
\bibitem{MM} MacLane, S., Moerdijk, I.: Sheaves in Geometry and
Logic. Springer-Verlag, New York (1992)
\bibitem{mccarty}  McCarty, D.: Realizability and recursive
mathematics. Ph.D. thesis, Oxford University (1984); also
Carnegie-Mellon University Technical Report CMU-CS-84-131
\bibitem{Mo} Mostowski, A.W.: Proofs of non-deducibility in
intuitionistic functional calculus. JSL 13, pp. 204-207 (1948)
\bibitem{Sc1} Scott, D.S.: Extending the topological
interpretation to intuitionistic analysis I. Compos. Math. 20, pp.
194-210 (1968)
\bibitem{Sc2} Scott, D.S.: Extending the topological
interpretation to intuitionistic analysis II. In Myhill, J., Kino,
A., Vesley, R.E. (eds.) Intuitionism and Proof Theory pp. 235-255.
North-Holland, Amsterdam (1970)
\bibitem{St} Stone, M.H.: Topological representations of
distributive lattices and Brouwerian logics. Casopis Pro
Pestvovani Mathematiky a Fysiki Cast Matematicka 67, pp. 1-25
(1937)
\bibitem{T} Tarski, A.: Der Aussagenkalk\"ul and die Topologie.
Fundam. Math. 31, pp. 103-134 (1938)
\bibitem{TvD} Troelstra, A.S., van Dalen, D.: Constructivism in
Mathematics -- An Introduction, Vol. II, Studies in Logic and the
Foundations of Mathematics, Vol. 123. North-Holland, Amsterdam New
York Oxford (1988)

\end{thebibliography}
\end{document}